\documentclass[11pt]{amsart}
\usepackage{amssymb}
\usepackage{amsmath}
\usepackage{amsfonts}
\usepackage{graphicx}

\usepackage[total={17cm,22cm},top=2.5cm, left=2.3cm]{geometry}
\usepackage{hyperref}
    \usepackage{aeguill}
    \usepackage{type1cm}

\theoremstyle{plain}
\newtheorem{thm}{Theorem}[section]

\newtheorem{claim}[thm]{Claim}

\newtheorem{definition}[thm]{Definition}
\newtheorem{example}[thm]{Example}

\newtheorem{fact}[thm]{Fact}
\newtheorem{lemma}[thm]{Lemma}

\newtheorem{problem}[thm]{Problem}
\newtheorem{proposition}[thm]{Proposition}
\newtheorem{remark}[thm]{Remark}

\newtheorem{theorem}[thm]{Theorem}
\numberwithin{equation}{section}
\newcommand{\N}{\mathbb{N}}

\newcommand{\R}{\mathbb{R}}

\DeclareMathOperator{\lip}{Lip\,\!}

\DeclareMathOperator{\sign}{sign\,\!}

\DeclareMathOperator{\dist}{dist\,\!}

\DeclareMathOperator{\sop}{supp\,\!}

\DeclareMathOperator{\diam}{diam\,\!}

\usepackage[usenames,dvipsnames]{color}

\usepackage[dvipsnames]{xcolor}

\begin{document}

\title[Approximation of Lipschitz functions preserving boundary values]{Approximation of Lipschitz functions preserving boundary values}

\author{Robert Deville}
\address{Institut de Math\'{e}matiques de Bordeaux, Universit\'{e} de Bordeaux 1, 33405, Talence, France }
\email{Robert.Deville@mat.u-bordeaux.fr}

\author{Carlos Mudarra}
\address{Instituto de Ciencias Matem\'{a}ticas (CSIC-UAM-UC3-UCM), 28049 Madrid, Spain}
\email{carlos.mudarra@icmat.es}

\date{Jan 23, 2018}


\thanks{C. Mudarra was supported by the Grants: ``Programa Internacional de Doctorado Fundaci\'on La Caixa--Severo Ochoa'' and MTM2015-65825-P}


\begin{abstract}
Given an open subset $\Omega$ of a Banach space and a Lipschitz function $u_0: \overline{\Omega}  \to \R,$ we study whether
it is possible to approximate $u_0$ uniformly on $\Omega$ by $C^k$-smooth Lipschitz functions which coincide with $u_0$ on the boundary $\partial \Omega$ of $\Omega$ and have the same Lipschitz constant as $u_0.$ 
As a consequence, we show that every $1$-Lipschitz function $u_0: \overline{\Omega} \to \R,$ defined on the closure
$\overline{\Omega}$ of an open subset $\Omega$ of a finite dimensional normed space of dimension $n \geq 2$, and such that
the Lipschitz constant of the restriction of $u_0$ to the boundary of $\Omega$ is less than $1$,
can be uniformly approximated by differentiable $1$-Lipschitz functions $w$ which coincide
with $u_0$ on $\partial \Omega$ and satisfy the equation $\| D w\|_* =1$ almost everywhere on $\Omega.$ This result does
not hold in general without assumption on the restriction of $u_0$ to the boundary of $\Omega$. 
\end{abstract}

\maketitle

\section{Introduction and main results}
Throughout this paper, for every metric space $(E,d)$ and every function $f:E \to \R,$ 
we will denote the Lipschitz constant of $f$ on $E$ by $\lip(f,E),$ that is,
$$
\lip(f,E):= \inf \lbrace L >0 \: : \: |f(x)-f(y)| \leq L d(x,y) \quad \text{for all} \quad x,y\in E \rbrace.
$$
Also, if $\lambda \geq 0,$ we will say that $f:E \to \R$ is $\lambda$-Lipschitz on $E$ 
whenever $|f(x)-f(y)| \leq \lambda d(x,y)$ for every $x,y\in E.$ We will denote by $B(x_0,r)$ the closed ball centered at $x_0$ and with radius $r>0$
with respect to the metric on $E.$ Finally, for any Banach space $X$ with norm $\| \cdot \|,$ the dual norm on $X^*$ will be denoted by $\| \cdot \|_*.$ 

\medskip

In this paper we deal with the following problem.

\begin{problem}\label{mainproblem}
Let $X$ be a Banach space, let $u_0: \overline{\Omega} \to \R$ be a Lipschitz function defined on the closure of an open subset $\Omega$ of $X$ and let $k \in \N \cup \lbrace \infty \rbrace.$ Given $\varepsilon >0,$ does there exist a function $v: \overline{\Omega} \to \R$ of class $C^k(\Omega)$ with $\lip(v, \overline{\Omega}) \leq  \lip(u_0, \overline{\Omega}), \: v=u_0 $ on $\partial \Omega$ and $|u_0-v| \leq \varepsilon$ on $\overline{\Omega}$ ?
\end{problem}

In finite dimensional spaces, the integral convolution with mollifiers provides uniform approximation by $C^\infty$ functions
preserving the Lipschitz constant of the function to be approximated. However this approximation does not necessarily preserve
the value of $u_0$ on $\partial \Omega.$ On the other hand, it was proved in \cite[Theorem 2.2]{CzarneckiRifford} an approximation theorem
for locally Lipschitz functions defined on open subsets of $\R^n$ which implies that for any continuous function $\delta:
\Omega \to (0,+ \infty),$ and any locally Lipschitz function $u_0$ there exists a function $v$ of class $C^\infty$ satisfying
(among other properties) that
$$
|u_0(x)-v(x)| \leq \delta(x) \quad \text{and} \quad | Dv(x)| \leq \lip( u_0 , B(x, \delta(x)) \cap \Omega ) + 
\delta(x), \quad x\in \Omega.
$$

Using the above result with $\delta(x)= \min \lbrace \varepsilon, \dist(x,\partial \Omega) \rbrace$ we get a smooth Lipschitz 
approximation $v$ of $u_0$ that extends continuously to $\overline{\Omega}$ by setting $v=u_0$ on $\partial \Omega.$ The 
function $v$ has Lipschitz constant arbitrarily close to $\lip(u_0, \overline{\Omega}),$ but bigger than $\lip(u_0, 
\overline{\Omega})$ in general. Thus this does not yield any answer to Problem \ref{mainproblem}.

\medskip

In the infinite dimensional case, it was proved in \cite[Theorem 1]{AFLR} that any Lipschitz function defined on an open subset $\Omega$ of a separable Hilbert space (or even a separable infinite dimensional Riemannian manifold) can be approximated in the $C^0$-fine topology by $C^\infty$ functions whose Lipschitz constant can be taken to be arbitrarily close to the Lipschitz constant of $u_0,$ i.e., for any given continuous function $\delta : \Omega \to (0,+ \infty)$ and $r>0,$ there exists $v$ of class $C^\infty$ such that
$$
|u_0(x)-v(x)| \leq \delta(x), \quad  x\in \Omega \quad \text{and} \quad \lip(v, \Omega) \leq \lip(u_0, \overline{\Omega})+r.
$$

\medskip

One can find in \cite{AFK, HajekJohanis, MJSLSG} some results on approximation of Lipschitz functions by $C^k$-smooth Lipschitz functions in more general Banach spaces. In these results, the approximating function preserves the Lipschitz constant of the original function up to a factor $C_0\geq 1,$ which only depends on the space and is bigger than $1$ in general.

\medskip

In this paper we show that the answer to Problem \ref{mainproblem} depends on the relation between $\lip(u_0, \partial \Omega)$ and $\lip(u_0, \overline{\Omega}).$ Let us now state our main results in this direction.

\begin{theorem}\label{secondmaintheorem}
Let $X$ be a finite dimensional normed space, or a separable Hilbert space or the space $c_0(\Gamma),$ for an arbitrary set of indices $\Gamma.$ Let $\Omega$ be an open subset of $X$ and let $u_0 : \overline{\Omega} \to \R$ be a Lipschitz function such that $\lip(u_0, \partial \Omega) < \lip(u_0, \overline{\Omega}).$ Given $\varepsilon >0,$ there exists a function $v: \overline{\Omega} \to \R$ such that $v$ is of class $C^\infty(\Omega), \: v$ is Lipschitz on $\overline{\Omega}$ with $\lip(v, \overline{\Omega})\leq \lip(u_0, \overline{\Omega}), \: v=u_0$ on $\partial \Omega$ and $|u_0-v| \leq \varepsilon$ on $\overline{\Omega}.$ 
\end{theorem}

For non-separable Hilbert spaces, we have the following. 

\begin{theorem}\label{maintheoremnonseparablehilbert}
Let $X$ be a Hilbert space. Let $\Omega$ be an open subset of $X$ and let $u_0 : \overline{\Omega} \to \R$ be a Lipschitz function such that $\lip(u_0, \partial \Omega) < \lip(u_0, \overline{\Omega}).$ Given $\varepsilon >0,$ there exists a function $v: \overline{\Omega} \to \R$ such that $v$ is of class $C^1(\Omega), \: v$ is Lipschitz on $\overline{\Omega}$ with $\lip(v, \overline{\Omega}) \leq \lip(u_0, \overline{\Omega}), \: v=u_0$ on $\partial \Omega$ and $|u_0-v| \leq \varepsilon$ on $\overline{\Omega}.$ 
\end{theorem}

Theorems \ref{secondmaintheorem} and \ref{maintheoremnonseparablehilbert} gives a positive answer to Problem \ref{mainproblem} for the $C^1(\Omega)$ or $C^\infty(\Omega)$ class, when $\lip(u_0, \partial \Omega) < \lip(u_0, \overline{\Omega}),$ in certain Banach spaces. These theorems will be proved by combining approximation techniques in the pertinent space with the following result.

\begin{theorem}\label{generaltheorem}
Let $k\in \N \cup \lbrace \infty \rbrace$ and let $X$ be a Banach space with the property that for every Lipschitz function $f: X \to \R$ and every $\eta>0,$ there exists a function $g :X \to \R$ of class $C^k(X)$ such that $|f-g| \leq \eta$ on $X$ and $\lip(g, B(x_0,r)) \leq \lip(f, B(x_0, r+\eta) ) + \eta$ for every ball $B(x_0,r) \subset X.$ Then, if $\Omega$ is an open subset of $X, \: u_0 : \overline{\Omega} \to \R$ is a Lipschitz function such that $\lip(u_0, \partial \Omega) < \lip(u_0, \overline{\Omega})$ and $\varepsilon >0,$ there exists a function $v: \overline{\Omega} \to \R$ such that $v$ is of class $C^k(\Omega), \: v$ is Lipschitz on $\overline{\Omega}$ with $\lip(v, \overline{\Omega}) \leq \lip(u_0, \overline{\Omega}), \: v=u_0$ on $\partial \Omega$ and $|u_0-v| \leq \varepsilon$ on $\overline{\Omega}.$ 
\end{theorem}

In Section \ref{sectionexample}, we will see an example on $\R^2$ with the $\ell_1$ norm showing that Problem \ref{mainproblem} has a negative answer (even for the class of functions which are merely differentiable on $\Omega$) if we allow $\lip(u_0, \partial \Omega) = \lip(u_0, \overline{\Omega}).$ Therefore, one can say that Theorem \ref{secondmaintheorem} is optimal (in the sense of Problem \ref{mainproblem}), at least in the setting of finite dimensional normed spaces.

\medskip

We now consider a subproblem of Problem \ref{mainproblem} when $X$ is a finite dimensional normed space.

\begin{problem}\label{mainproblemalmostclassicalsolutions}
Let $(X, \| \cdot \|)$ be a finite dimensional normed space with $\dim(X) \geq 2$ and let $u_0: \overline{\Omega} \to \R$ be a $1$-Lipschitz function defined on the closure of an open subset $\Omega$ of $X.$ Given $\varepsilon >0,$ does there exist a $1$-Lipschitz function $w: \overline{\Omega} \to \R$ such that $w$ is differentiable on $\Omega$ with $\| Dw\|_*=1$ almost everywhere on $\Omega, \: w=u_0 $ on $\partial \Omega$ and $|u_0-w| \leq \varepsilon$ on $\overline{\Omega}$ ?
\end{problem}

Observe that if $w=u_0$ on $\partial \Omega$ and $\lip(u_0,\partial \Omega) <1,$ then the Mean Value Theorem yields the existence of $x\in \Omega$ such that $\| Dw(x)\|_* <1.$ Therefore the function $w$ (if it exists) has no continuous derivative in this case.

\medskip

The following theorem gives a positive answer to Problem \ref{mainproblemalmostclassicalsolutions} when $\lip(u_0, \partial \Omega) <1.$ 
\begin{theorem}\label{maintheoremarbitrarynorm}
Let $\Omega$ be an open subset of a finite dimensional normed space $(X, \| \cdot \|)$ with $ \dim(X) \geq 2.$ Let $u_0: \overline{\Omega} \to \R$ be a $1$-Lipschitz function such that $\lip(u_0, \partial \Omega) <1.$ Given $\varepsilon >0,$ there exists a differentiable $1$-Lipschitz function $w: \overline{\Omega} \to \R$ such that $\| D w\|_*=1 $ almost everywhere on $\Omega, \: w=u_0$ on $\partial \Omega$ and $|u_0-w| \leq \varepsilon$ on $\overline{\Omega}.$
\end{theorem}

In Section \ref{sectionexample}, we prove, using the theory of almost minimizing Lipschitz extensions, 
that if $\Omega$ is an open subset in a $2$-dimensional euclidean space and
if $u_0: \partial \Omega \to \R$ is a $1$-Lipschitz function, then there exists a differentiable $1$-Lipschitz function 
$w:  \overline{\Omega} \to \R$ such that $\| D w\|_*=1 $ almost everywhere on $\Omega$ and $\: w=u_0$ on $\partial \Omega$. However, Example \ref{counterexamplel1} in Section \ref{sectionexample} shows that the above theorem is optimal in the sense of Problem \ref{mainproblemalmostclassicalsolutions}. Observe that Theorem \ref{maintheoremarbitrarynorm} covers the case of homogeneous Dirichlet conditions. Also, we notice that the above theorem does not hold when $X= \R.$ Indeed, if $u_0: [0,1] \to \R$ is $1$-Lipschitz and differentiable on $(0,1),$ with $|u_0(1)-u_0(0)| <1,$ then a result of A. Denjoy \cite{Denjoy} tells us that either $\lbrace x \: : \: |u_0'(x)| <1 \rbrace$ is empty or else it has positive Lebesgue measure. But this subset is nonempty by the Mean Value Theorem. 

\medskip

The contents of the paper are as follows. In Section \ref{sectionapproximationmetricspaces}, we show that in general metric spaces, one can approximate a Lipschitz function $u_0$ by a function which coincides with $u_0$ on a given subset and has, on bounded subsets, better Lipschitz constants. In Section \ref{sectionsmoothapproximation}, we will give the proof of Theorems \ref{generaltheorem}, \ref{secondmaintheorem} and \ref{maintheoremnonseparablehilbert} with the decisive help of the above result. In Section \ref{sectionapproximationalmostclassical}, we use Theorem \ref{secondmaintheorem} and the results in \cite{DevilleJaramillo} to prove Theorem \ref{maintheoremarbitrarynorm}. Finally, in Section \ref{sectionexample}, we consider the case $\lip(u_0, \partial \Omega) = \lip(u_0, \overline{\Omega}):$ although a partial positive result in the euclidean setting can be obtained, we show that Problem \ref{mainproblem} does not always have a positive answer in this limiting case.

\section{Approximation by functions with smaller Lipschitz constants}\label{sectionapproximationmetricspaces}
Throughout this section, all the sets involved are considered to be subsets of a metric space $(X,d)$ and all the Lipschitz constants are taken with respect to the distance $d.$ The following result will be very useful in Section \ref{sectionsmoothapproximation} and it is interesting in itself.

\begin{theorem}\label{theoremglobalapproximation}
Let $E$ and $F$ be two nonempty closed sets such that $F \subset E,$ let $u_0: E \to \R$ be a $K$-Lipschitz function such that $\lambda_0:=\lip(u_0,F) <K.$ Given $\varepsilon >0,$ there exists a function $u: E \to \R$ such that $|u-u_0| \leq \varepsilon$ on $E, \: u=u_0$ on $F$ and $u$ has the property that $\lip(u,B) <K$ for every bounded subset $B$ of $E.$
\end{theorem}

A crucial step for proving the above theorem is the following lemma. For any two nonempty subsets $A$ and $B$ of $X$
and for any $x\in X$, we will denote 
$$
\dist(x,B):= \inf \lbrace d(x,y) \: : \: y\in B \rbrace,
$$
$$
\dist(A,B):= \inf \lbrace d(x,y) \: : x\in A,\: y\in B \rbrace\quad\text{and}\quad
\diam(A):= \sup \lbrace d(x,y) \: : \: x,y\in A \rbrace.
$$
\begin{lemma}\label{lemmalocalapproximation}
Let $E$ and $F$ be two nonempty closed subsets such that $F \subset E$ and $E \setminus F$ is bounded. Let $u_0: E \to \R$ be a $1$-Lipschitz function, let $u_\mu : F \to \R$ be $\mu$-Lipschitz, with $\mu < 1,$ let $\delta \geq 0$ and assume that $| u_\mu -u_0| \leq \delta$ on $F.$ For every $ \mu < \lambda < 1,$ there exists a function $u_\lambda: E \to \R$ such that $u_\lambda$ is $\lambda$-Lipschitz on $E$ with $u_\lambda = u_\mu$ on $F$ and $| u_0- u_\lambda |\leq \delta + \varepsilon( \lambda, \mu, E, F)$ on $E;$ where
$$
\varepsilon( \lambda, \mu, E, F) = \frac{1-\lambda}{\lambda-\mu}(\lambda + \mu) 
\left( \diam( \overline{E\setminus F}) + \dist( \overline{E\setminus F}, F ) \right)>0 
$$
and $\varepsilon( \lambda, \mu, E, F) =0$ whenever $E \setminus F= \emptyset.$ 
\end{lemma}

\begin{proof}
In the case when $E \setminus F = \emptyset,$ we have that $E=F$ and then it is enough to take $u_\lambda = u_\mu.$ From now on, we assume that $E \setminus F\ne\emptyset$, we fix $\mu <\lambda < 1$,  and we denote $\varepsilon_\lambda = \varepsilon( \lambda, \mu, E, F)$. We now define the strategy of proof of the lemma. We first show that the family
$$
\mathcal{C}_\lambda := \lbrace u : E \to \R \: : \: u \:\:\text{is} \: \: \lambda\text{-Lipschitz on} \:\: E, \: u\leq u_0 + \delta+ \varepsilon_\lambda \:\: \text{on} \:\: E, \: u=u_\mu \:\: \text{on} \:\: F \rbrace
$$
is nonempty, and then we define the function $u_\lambda$ by:
\begin{equation} \label{definitionulambda}
u_\lambda (x) := \sup\lbrace u(x) \: : \: u \in \mathcal{C}_\lambda \rbrace, \quad x\in E.
\end{equation}
In order to prove that the function $u_\lambda$ is the required solution, it will be enough to check that 
$u_\lambda\in\mathcal{C}_\lambda$ and that $u_0 \leq u_\lambda + \delta + \varepsilon_\lambda$ on $E.$

\item[] $\textbf{1.}$ We now prove that the family $\mathcal{C}_\lambda$ is nonempty. Consider the function
$$
v(x)= \sup_{y\in F} \lbrace u_\mu(y)-\lambda d(x,y) \rbrace, \quad x\in E,
$$
and let us see that $v\in C_\lambda.$ Since $u_\mu$ is $\lambda$-Lipschitz (in fact, $\mu$-Lipschitz) on $F,$ it follows from standard calculations concerning the sup convolution of Lipschitz functions that $v$ is a well-defined $\lambda$-Lipschitz function on $E$ with $v=u_\mu$ on $F.$ Now, given $x \in E\setminus F$ and $y \in F$ let us see that $u_\mu(y)-\lambda d(x,y) \leq u_0(x) + \delta+ \varepsilon_\lambda.$ For every $\eta>0,$ we can find a point $z_\eta\in F$ with
\begin{equation}\label{pointminimizingdistance}
\dist(x,F ) + \eta \geq d(x,z_\eta).
\end{equation}
In the case when $u_\mu(y)-\lambda d(x,y) < u_\mu(z_\eta)-\lambda d(x,z_\eta),$ by the assumption that $|u_\mu-u_0| \leq \delta$ on $F$ together with \eqref{pointminimizingdistance} and the fact that 
$\dist(x,F)\le\varepsilon_\lambda$, we have that
\begin{align*}
 u_\mu(y)-\lambda d(x,y) & < u_\mu(z_\eta)-\lambda d(x,z_\eta) \leq u_0(z_\eta) +\delta - \lambda d(x,z_\eta) \leq u_0(x)+ \delta + (1-\lambda) d(x,z_\eta) \\
 & \leq u_0(x)+ \delta + (1-\lambda) \left( \dist(x,F ) + \eta \right) \leq u_0(x)+ \delta +\varepsilon_\lambda + (1-\lambda) \eta.
\end{align*}
In the case when $u_\mu(y)-\lambda d(x,y) \geq u_\mu(z_\eta)-\lambda d(x,z_\eta).$ The fact that $u_\mu$ is $\mu$-Lipschitz on $F$ yields
\begin{align*}
u_\mu(y)-\lambda d(x,y) & \geq u_\mu(z_\eta)-\lambda d(x,z_\eta) \geq u_\mu(y)-\mu d(y,z_\eta)-\lambda d(x,z_\eta) \\
& \geq u_\mu(y) -\mu d(x,y)- \mu d(x,z_\eta)-\lambda d(x,z_\eta),
\end{align*}
which in turn implies 
\begin{equation}\label{comparabledistance}
(\lambda- \mu) d(x,y) \leq (\lambda+ \mu) d(x,z_\eta).
\end{equation}
Using first that $u_0$ is $1$-Lipschitz on $E$ and then \eqref{comparabledistance} and \eqref{pointminimizingdistance}, we obtain
\begin{align*}
u_\mu(y)-\lambda d(x,y)  & \leq u_0(y)+ \delta -\lambda d(x,y) \leq u_0(x)+ \delta + (1-\lambda ) d(x,y) \\
& \leq u_0(x)+ \delta+ \frac{1-\lambda}{\lambda-\mu}(\lambda + \mu) d(x,z_\eta) \leq u_0(x)+ \delta+ \frac{1-\lambda}{\lambda-\mu}(\lambda + \mu) \left( \dist(x,F)+ \eta \right) \\
& \leq u_0(x)+ \delta+ \varepsilon_\lambda + \frac{1-\lambda}{\lambda-\mu}(\lambda + \mu) \: \eta.
\end{align*}
Hence, in both cases, we have that
$$
 u_\mu(y)-\lambda d(x,y)  \leq u_0(x)+ \delta+ \varepsilon_\lambda + \frac{1-\lambda}{\lambda-\mu}(\lambda + \mu) \: \eta,
 $$
and letting $\eta \to 0^+,$ it follows that $v(x) \leq u_0(x)+ \delta+ \varepsilon_\lambda$ for every $x\in \overline{E \setminus F}.$ This proves the inequality $v \leq u_0+ \delta+ \varepsilon_\lambda$ on $E,$ which shows that $v\in \mathcal{C}_\lambda.$ 

\medskip

\item[] $\textbf{2.}$ 
The function $u_\lambda$ belongs to $\mathcal{C}_\lambda$ because a supremum of $\lambda$-Lipschitz functions 
is a $\lambda$-Lipschitz function, and because inequalities and equalities are preserved by taking supremum. 
Before proving the inequality 
$u_0 \leq u_\lambda + \delta + \varepsilon_\lambda$ on $E$, 
we first show that $u_\lambda$ coincides with the function
$$
v_\lambda(x):= \inf_{y \in F \cup S_\lambda} \lbrace u_\lambda(y) + \lambda d(x,y) \rbrace, \quad x\in E;
$$ 
where
$$
S_\lambda= \left\lbrace x\in E  \: : \: u_\lambda(x) \geq u_0(x)+\delta+ \frac{\varepsilon_\lambda}{2} \right\rbrace.
$$
Observe that, since $u_\mu \leq u_0+ \delta$ on $F, \: S_\lambda$ and $F$ are disjoint. Since $u_\lambda$ is $\lambda$-Lipschitz on $E$ (and, in particular, on $F \cup S_\lambda$), the function $v_\lambda$ is the greatest $\lambda$-Lipschitz extension of $u_\lambda$ from the set $F \cup S_\lambda.$ Thus $v_\lambda = u_\lambda$ on $F \cup S_\lambda$ and $u_\lambda \leq v_\lambda$ on $E.$ Hence, by \eqref{definitionulambda}, we will have that $v_\lambda = u_\lambda$ as soon as we see that $v_\lambda \leq u_0 + \delta+ \varepsilon_\lambda$ on $E.$ Let us define
$$
G_\lambda = \lbrace x\in E \setminus \left( F \cup S_\lambda \right) \: : \: v_\lambda(x) \geq u_0(x) + \delta + \varepsilon_\lambda \rbrace.
$$
\begin{claim}\label{subsetemptyclaim}
$G_\lambda = \emptyset.$
\end{claim}
Assume that $G_\lambda \neq \emptyset.$ Since $E\setminus F$ is bounded, then $v_\lambda-u_0$ is bounded on $G_\lambda$ and we can define 
$$
a:= \sup_{G_\lambda} \lbrace v_\lambda - u_0 \rbrace.
$$ It is obvious that $a \geq  \delta + \varepsilon_\lambda.$ We can pick a point $y \in G_\lambda$ such that
\begin{equation}\label{approximationsupremumglambda}
 v_\lambda(y)-u_0(y) \geq a - \frac{\varepsilon_\lambda}{2}.
\end{equation}
We next define the function
$$
w_\lambda : = \max \lbrace u_\lambda ,  v_\lambda -a + \delta +  \varepsilon_\lambda \rbrace  : E \to \R.
$$
The function $w_\lambda$ is $\lambda$-Lipschitz on $E$ and satisfies the following.
\item[] $(i)$ On the set $ F \cup S_\lambda,$ we have $v_\lambda= u_\lambda.$ Since $a \geq  \delta + \varepsilon_\lambda,$ we have that $w_\lambda = u_\lambda$ on $F \cup S_\lambda.$ In particular $w_\lambda = u_\mu$ on $F.$ 
\item[] $(ii)$ On $G_\lambda,$ we have, by the definition of $a,$ that 
$
v_\lambda -a  \leq  u_0.
$
Since we always have $u_\lambda \leq u_0 + \delta +  \varepsilon_\lambda,$ the function $w_\lambda$ satisfies $w_\lambda \leq u_0 +   \delta + \varepsilon_\lambda$ on $G_\lambda.$
\item[] $(iii)$ If $x\in E \setminus  (G_\lambda \cup F \cup S_\lambda),$ then
$$
v_\lambda (x)-a < u_0(x)+ \delta + \varepsilon_\lambda - a  \leq u_0(x),
$$
together with $u_\lambda \leq u_0 + \delta+  \varepsilon_\lambda$ on $E,$ this implies $w_\lambda(x) \leq u_0(x) + \delta + \varepsilon_\lambda.$

\medskip

From the remarks $(i), (ii)$ and $(iii)$ above we obtain that $w_\lambda \leq u_0 + \delta+  \varepsilon_\lambda$ on $E$ with $w_\lambda = u_\mu$ on $F.$ By \eqref{definitionulambda} we must have $w_\lambda \leq u_\lambda$ on $E.$ But, for the point $y \in G_\lambda,$ (see \eqref{approximationsupremumglambda}) it follows that
$$
u_\lambda(y)  \geq  w_\lambda(y) \geq v_\lambda(y)-a+ \delta + \varepsilon_\lambda \geq u_0(y)+ \delta + \frac{\varepsilon_\lambda}{2} .
$$
It turns out that $y$ belongs to $S_\lambda,$ which is a contradiction since $G_\lambda$ and $S_\lambda$ are disjoint subsets. This proves Claim \ref{subsetemptyclaim}.

\medskip

Finally, because $G_\lambda = \emptyset,$ it is clear that $v_\lambda \leq u_0 + \delta+ \varepsilon_\lambda$ on $E$ and therefore
\begin{equation}\label{ulambdaequalinfimalconvolution}
u_\lambda (x) = v_\lambda (x) = \inf_{y\in F \cup S_\lambda} \lbrace u_\lambda(y) + \lambda d(x,y) \rbrace, \quad x\in E.
\end{equation}

\medskip

\item[] $\textbf{3.}$ We now show that $u_0(x) \leq u_\lambda(x) + \delta + \varepsilon_\lambda$ for every $x\in E.$ Since $u_0 \leq u_\mu + \delta =u_\lambda+ \delta$ on $F,$ we only need to consider the situation when $x\in E \setminus F.$ Let us fix $\eta >0.$ We can find a point $z_\eta \in F$ with
\begin{equation}\label{pointminimizingdistance2}
\dist(x,F) + \eta \geq d(x,z_\eta).
\end{equation}
Moreover, by \eqref{ulambdaequalinfimalconvolution}, it is clear that there exists $y_\eta \in F \cup S_\lambda$ such that
\begin{equation}\label{sequenceynminimizing}
u_\lambda(y_\eta) + \lambda d(x,y_\eta) \leq \min \left\lbrace u_\lambda(z_\eta) + \lambda d(x,z_\eta), u_\lambda(x)+\eta \right\rbrace.
\end{equation}
Suppose first that $y_\eta \in S_\lambda.$ In particular $y_\eta \in E \setminus F$ and $u_\lambda(y_\eta)\geq u_0(y_\eta)+ \delta+ \frac{\varepsilon_\lambda}{2}.$ Using that $u_0$ is $1$-Lipschitz together with \eqref{sequenceynminimizing} we obtain
\begin{align*}
u_0(x) & \leq u_0(y_\eta) + d(x,y) =  u_0(y_\eta) + \lambda d(x,y_\eta) + (1- \lambda) d(x,y_\eta) \\
& \leq u_\lambda(y_\eta)-\delta- \frac{\varepsilon_\lambda}{2} + \lambda d(x,y_\eta)+ (1- \lambda) d(x,y_\eta) \\
& \leq u_\lambda(x)+\eta -\delta- \frac{\varepsilon_\lambda}{2}+ (1-\lambda) \diam( \overline{E \setminus F}) \leq u_\lambda(x) + \delta +\varepsilon_\lambda + \eta.
\end{align*}
Suppose now that $y_\eta \in F.$ Using \eqref{sequenceynminimizing} and the fact that $u_\lambda$ is $\mu$-Lipschitz on $F,$ we can write
\begin{align*}
u_\lambda(z_\eta) + \lambda d(x,z_\eta) & \geq u_\lambda (y_\eta) + \lambda d(x,y_\eta) \geq u_\lambda(z_\eta)- \mu d(y_\eta,z_\eta)+\lambda d(x,y_\eta) \\
& \geq u_\lambda(z_\eta)- \mu d(x,z_\eta)+(\lambda - \mu) d(x,y_\eta) ,
\end{align*}
which implies, taking into account \eqref{pointminimizingdistance2},
\begin{equation}\label{comparabledistances2}
d(x,y_\eta) \leq \frac{\lambda + \mu}{\lambda-\mu} d(x,z_\eta) \leq \frac{\lambda + \mu}{\lambda-\mu} (\dist(x,F)+\eta) \leq \frac{\varepsilon_\lambda}{1-\lambda} + \frac{\lambda + \mu}{\lambda-\mu} \: \eta.
\end{equation}
Bearing in mind that $u_\lambda + \delta = u_\mu+ \delta \geq u_0$ on $F$ and using \eqref{sequenceynminimizing} and \eqref{comparabledistances2} we obtain
\begin{align*}
& u_0(x)  \leq  u_0(y_\eta) + \lambda d(x,y_\eta) + (1- \lambda) d(x,y_\eta) \\
& \leq u_\lambda(y_\eta) + \delta + \lambda d(x,y_\eta) + (1- \lambda) d(x,y_\eta) \leq u_\lambda(x) + \eta + \delta+ \varepsilon_\lambda + (1-\lambda) \frac{\lambda + \mu}{\lambda-\mu} \: \eta.
\end{align*}
We have thus shown the inequality
$$
u_0(x) \leq u_\lambda(x) + \delta + \varepsilon_\lambda + \eta + (1-\lambda) \frac{\lambda + \mu}{\lambda-\mu} \: \eta \quad \text{on} \quad E.
$$
Letting $\eta \to 0^+,$ we conclude that $u_0(x) \leq u_\lambda(x) + \delta + \varepsilon_\lambda$ for every $x\in E.$
\end{proof}

\medskip

\begin{proof}[Proof of Theorem \ref{theoremglobalapproximation}]
Without loss of generality we may and do assume that $K=1.$ Let us fix a point $p\in  F$ and set $E_n = \left( E \cap B(p,n) \right) \cup  F$ and $F_n=E_{n-1}$ for every $n \geq 1,$ where $F_1 = E_0 = F.$ It is clear that we can construct an increasing sequence of numbers $\lbrace \lambda_n \rbrace_{n \geq 1}$ with $\lambda_0 < \lambda_1$ and $\lambda_n  <1$ for every $n \geq 1$ such that
\begin{equation}\label{inequalitychoicesequence}
\frac{1-\lambda_n}{\lambda_n-\lambda_{n-1}}(\lambda_n + \lambda_{n-1}) \left( \diam( \overline{E_n\setminus F_n}) + \dist( \overline{E_n\setminus F_n}, F_n ) \right) \leq \frac{\varepsilon}{2^n}
\end{equation}
for every $n \geq 1$ such that $E_n \setminus F_n \neq \emptyset.$ Let us construct by induction a sequence of functions $\lbrace u_n \rbrace_{n \geq 1}$ such that each $u_n : E_n \to \R$ is $\lambda_n$-Lipschitz on $E_n$ and satisfy $u_n=u_{n-1}$ on $E_{n-1}$ and $|u_n-u_0| \leq \varepsilon$ on $E_n$ for every $n \geq 1.$

\medskip
Since $u_0|_F$ is $\lambda_0$-Lipschitz, we can apply Lemma \ref{lemmalocalapproximation} with $F_1 \subset E_1, \: \delta=0, \: u_0: E_1 \to \R, \: \mu = \lambda_0, \: u_\mu =u_0|_{F_1}$ in order to obtain a $\lambda_1$-Lipschitz function $u_1: E_1 \to \R$ such that $u_1=u_\mu =u_0$ on $F_1$ and
$ |u_1-u_0| \leq \frac{\varepsilon}{2}$ on $ E_1,$ thanks to \eqref{inequalitychoicesequence}. Observe that $u_1=u_0$ on $F.$ Now assume that we have constructed functions $u_1, \ldots, u_n$ respectively defined on $E_1, \ldots, E_n$ such that each $u_k$ is $\lambda_k$-Lipschitz on $E_k,$ with $u_k=u_{k-1}$ on $E_{k-1}=F_k$ and 
$$
| u_k- u_0|  \leq \frac{\varepsilon}{2}+ \cdots + \frac{\varepsilon}{2^k} \quad \text{on} \quad E_k,
$$
for every $1 \leq k \leq n.$ Then we apply Lemma \ref{lemmalocalapproximation} with $\delta = \varepsilon/2+ \cdots + \varepsilon/2^n, \: E_n=F_{n+1} \subset E_{n+1}, \: \mu = \lambda_n, \: u_\mu = u_n: E_n \to \R$ and $u_0: E_{n+1} \to \R$ to obtain a $\lambda_{n+1}$-Lipschitz function $u_{n+1}: E_{n+1} \to \R$ such that $u_{n+1}=u_0$ on $E_n$ and, thanks to \eqref{inequalitychoicesequence},
$$
| u_{n+1}- u_0|  \leq \frac{\varepsilon}{2}+ \cdots + \frac{\varepsilon}{2^{n+1}}  \quad \text{on} \quad E_{n+1}.
$$
This proves the induction. We now define the function $u: E \to \R$ as follows: given $x\in E,$ we take a positive integer $n$ with $x\in E_n$ and set $u(x):=u_n(x).$ Since $E = \bigcup_{n \geq 1} E_n$ and each $u_n$ coincides with $u_{n-1}$ on $E_{n-1},$ the function $u$ is well defined. Because $u=u_n$ on each $E_n,$ we have that 
$$
|u-u_0| = |u_n-u_0| \leq \varepsilon \quad \text{on} \quad E_n,
$$
which implies that $|u-u_0| \leq \varepsilon$ on $E.$ Also, note that $u=u_0$ on $F$ because $u=u_1$ on $E_1$ and $u_1=u_0$ on $F \subset E_1.$ Finally, given a bounded subset $B$ of $E,$ we can find some natural $n$ with $B \subset E_n.$ This implies that $u=u_n$ on $B,$ where $u_n$ is $\lambda_n$-Lipschitz and $\lambda_n <1.$
\end{proof}

\section{Approximation by smooth Lipschitz functions: Proof of Theorem \ref{generaltheorem}}\label{sectionsmoothapproximation}
This section contains the proofs of Theorems \ref{generaltheorem}, \ref{secondmaintheorem} and \ref{maintheoremnonseparablehilbert}. Let us start with the proof of Theorem \ref{generaltheorem}, so let us assume from now on that $X$ is a Banach space satisfying the hypothesis of Theorem \ref{generaltheorem} for some $k\in \N \cup \lbrace \infty \rbrace.$ We will need to use the following two claims.

\begin{claim}\label{c0fineapproximation}
Let $\Omega \subset X$ be an open subset and let $u : \Omega \to \R$ be a Lipschitz function. For every continuous function $\varepsilon : \Omega \to (0,+ \infty)$ there exists $v : \Omega \to \R$ of class $C^k(\Omega)$ such that 
\item[] $(1)$ $|u(x)-v(x)| \leq \varepsilon(x)$ for all $x\in \Omega.$
\item[] $(2)$ $\| D v(x) \|_* \leq \lip( u , B(x, \varepsilon(x)) \cap \Omega ) + \varepsilon(x)$ for all $x\in \Omega.$ 
\end{claim}
\begin{proof}
By replacing $\varepsilon$ with $\min\lbrace \varepsilon, \frac{1}{2} \dist( \cdot, \partial \Omega) \rbrace,$ we may and do assume that $\varepsilon \leq \frac{1}{2} \dist( \cdot, \partial \Omega)$ on $\Omega,$ which implies that $B(x, \varepsilon(x))$ is contained in $\Omega$ for every $x\in \Omega.$ By continuity of $\varepsilon,$ for each $p \in \Omega,$ there exists $0< \delta_p \leq \varepsilon(p)/4$ such that $\varepsilon(x) \geq  \varepsilon(p)/2$ for all $x\in B(p, \delta_p).$ The assumption on $X$ implies in particular that there exists a constant $C_0 \geq 1$ such that, for every Lipschitz function $f: X \to \R$ and every $\eta >0,$ there exists a $C^k$ Lipschitz function $g: X \to \R$ such that $|f-g| \leq \eta$ on $X$ and $\lip(g,X) \leq C_0 \lip(f,X).$ Then, as a consequence of \cite[Lemma 3.6]{MJSLSG}, there exists a partition of unity $\lbrace \varphi_{n,p} \rbrace_{(n,p)\in \N \times \Omega}$ of class $C^k(\Omega)$ and Lipschitz such that $\sop(\varphi_{n,p}) \subset B(p,\delta_p)$ for every $(n,p) \in \N \times \Omega,$ and for every $x\in \Omega,$ there exists an open neighbourhood $U_x$ of $x$ and a positive integer $n_x$ such that
\begin{align}\label{locallyfiniteness}
& \text{If} \quad n >n_x, \quad \text{then} \quad U_x \cap \sop(\varphi_{n,p})= \emptyset \quad \text{for every} \: \: p\in \Omega. \\
& \text{If} \quad n \leq n_x, \quad \text{then} \quad U_x \cap \sop(\varphi_{n,p}) \neq  \emptyset \quad \text{for at most one} \: \: p\in \Omega. \nonumber
\end{align}

We can assume that $u$ is extended to all of $X$ with the same Lipschitz constant. Using the assumption on $X,$ we can find a family of $C^k(X)$ Lipschitz functions $\lbrace v_{n,p} \rbrace_{(n,p)\in \N \times \Omega}$ such that, for every $(n,p) \in \N \times \Omega,$
\begin{equation}\label{approximationgivenbylemma}
| u-v_{n,p}| \leq \frac{\varepsilon(p)}{(1+\lip(\varphi_{n,p}))2^{n+2}} \quad \text{on} \quad X \quad \text{and}
\end{equation}
\begin{equation}\label{preservinglocallipschitzconstant5}
\lip(v_{n,p}, B(x_0,r)) \leq \lip(u, B(x_0, r+\delta_p)) + \delta_p \leq \lip(u, B(x_0, r+ \delta_p)) + \frac{\varepsilon(p)}{4}
\end{equation}
for every ball $B(x_0,r)$ contained in $\Omega.$ We define the approximation $v: \Omega \to \R$ by
$$
v(x)= \sum_{(n,p) \in \N \times \Omega } v_{n,p}(x) \varphi_{n,p}(x), \quad x\in \Omega.
$$
By the properties of the partition $\lbrace \varphi_{n,p} \rbrace_{(n,p)\in \N \times \Omega},$ the function $v$ is well defined and is of class $C^k(\Omega).$ Given $x\in \Omega,$ \eqref{approximationgivenbylemma} implies
\begin{align*}
|u(x)-v(x)| &  \leq \sum_{ \lbrace (n,p) \: : \: B(p,\delta_p) \ni x \rbrace } | u(x)- v_{n,p}(x)| \: \varphi_{n,p}(x)  \leq \sum_{ \lbrace (n,p) \: : \: B(p,\delta_p) \ni x \rbrace } \frac{\varepsilon(p)}{2} \: \varphi_{n,p}(x) \\
& \leq \sum_{ \lbrace (n,p) \: : \: B(p,\delta_p) \ni x \rbrace }  \varepsilon(x) \:  \varphi_{n,p}(x) = \varepsilon(x).
\end{align*}

This proves part $(1)$ of our claim. Now, let us estimate $\|Dv(x)\|_*.$ Since $\sum_{(n,p)} \varphi_{n,p}=1,$ we have that $\sum_{(n,p)} D \varphi_{n,p}=0$ on $ \Omega.$ Then, taking into account that $\sop(\varphi_{n,p}) \subset B(p, \delta_p)$ for every $(n,p) \in \N \times \Omega,$ we can write
$$
Dv(x)= \sum_{ \lbrace (n,p) \: : \: B(p,\delta_p) \ni x \rbrace } D v_{n,p} (x) \varphi_{n,p} (x) + \sum_{ \lbrace (n,p) \: : \: B(p,\delta_p) \ni x \rbrace } (v_{n,p}(x)-u(x)) D \varphi_{n,p}(x).
$$
Hence, \eqref{approximationgivenbylemma} together with \eqref{locallyfiniteness} lead us to
\begin{align*}
\| Dv(x)\|_* & \leq \sum_{ \lbrace (n,p) \: : \: B(p,\delta_p) \ni x \rbrace } \|D v_{n,p} (x)\|_* \: \varphi_{n,p} (x) + \sum_{ \lbrace (n,p) \: : \: \varphi_{n,p}(x) \neq 0 \rbrace } \frac{\varepsilon(p)}{(1+\lip(\varphi_{n,p}))2^{n+2}} \|D \varphi_{n,p}(x)\|_* \\
& \leq \sum_{ \lbrace (n,p) \: : \: B(p,\delta_p) \ni x \rbrace } \|D v_{n,p} (x)\|_* \: \varphi_{n,p} (x) +\frac{\varepsilon(x)}{2} .
\end{align*}
Note that if $p \in \Omega$ is such that $x\in B(p,\delta_p),$ then $\varepsilon(x) \geq \varepsilon(p)/2 \geq 2 \delta_p$ and we can write, by virtue of \eqref{preservinglocallipschitzconstant5}, that
$$
\|Dv_{n,p}(x)\|_* \leq \lip( v_{n,p}, B(x, \varepsilon(x)-\delta_p)) \leq \lip( u, B(x, \varepsilon(x))) + \frac{\varepsilon(p)}{4} \leq \lip( u, B(x, \varepsilon(x))) + \frac{\varepsilon(x)}{2}.
$$
Therefore, we obtain
$$
\| Dv(x)\|_* \leq \sum_{ \lbrace (n,p) \: : \: B(p,\delta_p) \ni x \rbrace } \left( \lip( u, B(x, \varepsilon(x))) + \frac{\varepsilon(x)}{2} \right) \varphi_{n,p}(x) +\frac{\varepsilon(x)}{2} = \lip( u, B(x, \varepsilon(x))) + \varepsilon(x).
$$
This completes the proof of statement $(2).$ 

\end{proof}

\begin{claim}\label{propositionapproximationglobally1lipschitz}
Let $\Omega \subset X$ be an open subset and let $u: \Omega \to \R$ be a $K$-Lipschitz function with the property that $\lip(u,B) <K$ for every bounded subset $B$ of $\Omega.$ Then, given a continuous function $\varepsilon: \Omega \to (0,+ \infty),$ there exists $v: \Omega \to \R$ of class $C^k(\Omega)$ such that 

\item[] $(1)$ $|u(x)-v(x)| \leq \varepsilon(x)$ for every $x\in \Omega.$ 

\item[] $(2)$ $\| D v (x)\|_* <K$ for all $x\in \Omega.$ 
\end{claim}
\begin{proof}
Let us define $L(r) = \lip(u, B(0, r+1) \cap \Omega)$ for every $r>0.$ The function given by $\delta(r)= \frac{K-L(r)}{2},$ for every $r \geq 0,$ is positive and nonincreasing. The function $\tilde{\delta} : [0,+ \infty) \to \R$ given by
$$
\tilde{\delta}(t) = \int_{t}^{t+1} \delta(s) ds, \quad t\geq 0,
$$
is continuous and satisfies $\tilde{\delta} \left( [0,+ \infty) \right) \subset (0,K)$ and $\tilde{\delta} \leq \delta$ on $[0,+\infty).$ Let us define the mapping $ \rho : \Omega \to (0,+\infty)$ by $\rho(x)= \tilde{\delta}( \| x\|)$ for every $x\in \Omega.$ Then $\rho$ is continuous and we can replace $\varepsilon$ by $ \min \lbrace  1, \varepsilon,\rho, \frac{1}{2}\dist(\cdot, \partial \Omega) \rbrace$ on $\Omega.$ In particular, this implies that $B(x,\varepsilon(x)) \subset \Omega$ for every $x\in \Omega.$ We thus have from Claim \ref{c0fineapproximation} that there exists $v \in C^k(\Omega)$ such that 
$$
|u(x)-v(x)| \leq \varepsilon(x), \quad x\in \Omega,
$$
and
$$
\| Dv(x)\|_* \leq \lip(u, B(x, \varepsilon(x)) ) + \varepsilon(x), \quad x\in \Omega.
$$
Since $\varepsilon \leq 1,$ the ball $B(x, \varepsilon(x))$ is contained in $B(0, \| x\| + 1) \cap \Omega.$ Hence, the last inequality leads us to
$$
\| Dv(x)\|_* \leq L( \|x\| )  + \varepsilon(x) \leq L( \|x\| )  + \rho(x) \leq \frac{K+L( \|x\|)}{2}
$$
for every $x\in \Omega.$ This shows that $\| Dv(x)\|_*<K$ on $\Omega.$ 
\end{proof}

\medskip

We are now ready to prove Theorem \ref{generaltheorem}.

\begin{proof}[Proof of Theorem \ref{generaltheorem}]
Assume that $X$ satisfies the hypothesis of Theorem \ref{generaltheorem} for some $k\in \N \cup \lbrace \infty \rbrace.$ Let us denote by $\lambda_0$ and $K$ the Lipschitz constants $\lip(u_0, \partial \Omega)$ and $\lip(u_0, \overline{\Omega})$ of $u_0$ on $\partial \Omega$ and $\overline{\Omega}$ respectively. By Theorem \ref{theoremglobalapproximation}, there exists a function $u : \overline{\Omega} \to \R$ with 
\begin{equation}\label{estimationauxiliarfunction1}
|u_0-u| \leq \varepsilon/2 \quad \text{on} \quad \overline{\Omega} , \quad u=u_0 \quad \text{on} \quad \partial \Omega,
\end{equation}
and the Lipschitz constant of $u$ on every bounded subset of $\overline{\Omega}$ is strictly smaller than $K.$ Now, applying Claim \ref{propositionapproximationglobally1lipschitz} for $u,$ we can find a function $v: \Omega \to \R$ of class $C^k(\Omega)$ such that 
\begin{equation}\label{estimationauxiliarfunction2}
|u(x)-v(x)| \leq \min \left\lbrace \frac{\varepsilon}{2}, \dist( x, \partial \Omega) \right\rbrace \quad \text{and} \quad \| Dv(x)\|_* <K \quad \text{for all} \quad x\in \Omega.
\end{equation}
If we extend $v$ to the boundary $\partial \Omega$ of $\Omega$ by setting $v= u$ on $\partial \Omega$ and we use the inequality \eqref{estimationauxiliarfunction2}, we obtain, for every $x\in \partial \Omega, \: y\in \Omega,$ that
$$
|v(x)-v(y)| \leq |u(x)-u(y)| + |u(y)-v(y)| \leq K \|x-y\| + \dist(y, \partial \Omega) \leq (1+K) \|x-y\|.
$$
This proves that the function $v$ is continuous on $\overline{\Omega}.$ Therefore, the fact that $v$ is $K$-Lipschitz on $\overline{\Omega}$ is a consequence of the following well-known fact.

\begin{fact}\label{factlipschitzconstant}
{\em
If $w: \overline{\Omega} \to \R$ is continuous on $\overline{\Omega}$, is differentiable on $\Omega,$ is $K$-Lipschitz on $\partial \Omega$ and satisfies $\| Dw(x)\|_* \leq K$ for every $x\in \Omega,$ then $w$ is $K$-Lipschitz on $\overline{\Omega}.$ }
\end{fact}

It only remains to see that $v$ is $\varepsilon$-close to $u_0.$ Indeed, by using \eqref{estimationauxiliarfunction1} and \eqref{estimationauxiliarfunction2} we obtain 
$$
|u_0-v| \leq |u_0-u| + |u-v| \leq \frac{\varepsilon}{2} + \frac{\varepsilon}{2} = \varepsilon \quad \text{on} \quad \overline{\Omega}.
$$
\end{proof}

\subsection{Finite dimensional and Hilbert spaces}

We are now going to prove that if $X$ is a finite dimensional space or a Hilbert space, then $X$ satisfies the assumption of Theorem \ref{generaltheorem} with $k=\infty$ in the separable case and with $k=1$ in the non-separable case. 

\begin{lemma}\label{lemmapartialsmoothapproximation}
Let $X$ be a separable Hilbert space or a finite dimensional normed space. Given a $K$-Lipschitz function $f: X \to \R$ and $\varepsilon >0,$ there exists a function $g$ of class $C^\infty(X)$ such that $|g-f| \leq \varepsilon$ on $X$ and $\lip(g, B(x_0,r))  \leq \lip(f, B(x_0,r+ \varepsilon))+ \varepsilon$ for every ball $B(x_0,r) \subset X.$ On the other hand, if $X$ is a non-separable Hilbert space, the statement holds replacing $C^\infty$ smoothness with $C^1.$ 
\end{lemma}
\begin{proof}
Let us first consider that $X=\R^d$ is endowed with an arbitrary norm. If $f: \R^d \to \R$ is Lipschitz and we consider a function $\theta_\delta : \R^d \to \R$ of class $C^\infty(\R^d)$ with $\sop(\theta_\delta) \subseteq B(0,\delta)$ and $\int_{\R^d} \theta_\delta = 1,$ it is well known that the integral convolution $ f_\delta = f* \theta_\delta$ is a Lipschitz function of class $C^\infty$ such that
$$
\lip(f_\delta, S) \leq \lip(f, S+ B(0,\delta)) \quad \text{for every subset} \quad S \subset \R^d.
$$
In addition, $f_\delta \to f$ uniformly on $\R^d$ as $\delta \to 0^+.$ This proves the lemma in the finite dimensional case.

\medskip

Now, let $X$ be a Hilbert space and let us denote by $\| \cdot \|$ the norm on $X.$ If $g: X \to \R$ is a $K$-Lipschitz function, then the functions defined by
$$
g_\lambda(x)= \inf_{y\in X} \lbrace f(y) + \tfrac{1}{2\lambda}\| x-y\|^2 \rbrace, \quad g^\mu(x)= \sup_{y\in X} \lbrace f(y) - \tfrac{1}{2\mu}\| x-y\|^2 \rbrace
$$
for all $x\in X$ and $ \lambda, \mu >0,$ are $K$-Lipschitz as well. Also, it is easy to see that the infimum/supremum defining $g_\lambda(x)$ and $g^\mu(x)$ can be restricted to the ball $B(x, 2 \lambda K)$ and $B(x,2 \mu K)$ respectively. Let us now prove the following relation between the local Lipschitz constants of $g$ and $g_\lambda:$
\begin{equation}\label{preservinglocallipschitzconstant1}
\lip(g_\lambda , B(x_0,r)) \leq \lip(g, B(x_0,r+2\lambda K)) \quad \text{for every ball} \quad B(x_0,r) \subset X.
\end{equation}
Indeed, let us fix a ball $B(x_0,r),$ two points $x,x'\in B(x_0,r)$ and $\varepsilon >0.$ We can find $y \in B(x', 2 \lambda K)$ such that
$$
g(y)+\tfrac{1}{2\lambda}\| x'-y\|^2 \leq g_\lambda(x') + \varepsilon.
$$
The points $y$ and $x-x'+y$ belong to $B(x_0,r+ 2\lambda K)$ and then we can write
\begin{align*}
g_\lambda(x)-g_\lambda(x') & \leq g(x-x'+y) + \tfrac{1}{2\lambda}\| x-(x-x'+y)\|^2-g(y)\\
& \quad - \tfrac{1}{2\lambda} \| x'-y\|^2+ \varepsilon \leq \lip(g, B(x_0,r+2\lambda K)) \| x-x'\| + \varepsilon,
\end{align*}
which easily implies \eqref{preservinglocallipschitzconstant1}. Similarly, we show that 
\begin{equation}\label{preservinglocallipschitzconstant2}
\lip(g^\mu , B(x_0,r)) \leq \lip(g, B(x_0,r+2\mu K)) \quad \text{for every ball} \quad B(x_0,r) \subset X.
\end{equation}
Now, we consider the Lasry-Lions sup-inf convolution formula for $g,$ that is
$$
g_\lambda^\mu(x)= \sup_{z\in X} \inf_{y\in X} \lbrace f(y)+ \tfrac{1}{2\lambda} \| z-y\|^2- \tfrac{1}{2 \mu} \| x-z\|^2 \rbrace 
$$
for all $x\in X$ and $0 < \mu < \lambda.$ 
By the preceding remarks, the function $g_\lambda^\mu$ is $K$-Lipschitz and satisfies that
\begin{equation}\label{preservinglocallipschitzconstant3}
\lip(g_\lambda^\mu , B(x_0,r)) \leq \lip(g, B(x_0,r+2(\lambda + \mu) K)) \quad \text{for every ball }  B(x_0,r) \subset X.
\end{equation}
Moreover, in \cite{LasryLions, ATAZ} it is proved that $g_\lambda^\mu$ is of class $C^1(X)$ and $g_\lambda^\mu$ converges uniformly to $g$ as $0 < \mu < \lambda \to 0.$ Now, given our $K$-Lipschitz function $f: X \to \R$ and $\varepsilon >0,$ we can find $0 < \mu < \lambda$ small enough so that the function $f_\lambda^\mu$ is $K$-Lipschitz and of class $C^1(X), \: | f_\lambda^\mu-f| \leq \varepsilon/2$ on $X$ and, by virtue of \eqref{preservinglocallipschitzconstant3},
\begin{equation}\label{preservinglocallipschitzconstant4}
\lip(f_\lambda^\mu, B(x_0,r)) \leq \lip(f, B(x_0,r+ \varepsilon))\quad \text{for every ball }  B(x_0,r) \subset X.
\end{equation}
If we further assume that $X$ is separable, then we can use  \cite[Theorem 1]{Moulis} in order to obtain a function $g \in C^\infty(X)$ such that
$$
| f_\lambda^\mu - g| \leq \frac{\varepsilon}{2} \quad \text{and} \quad \| D f_\lambda^\mu - Dg\|_* \leq \varepsilon \quad \text{on} \quad X,
$$
where $\| \cdot \|_*$ denotes the dual norm of $\| \cdot \|.$ From the first inequality we see that $| f- g| \leq \varepsilon$ on $X.$ The second one together with \eqref{preservinglocallipschitzconstant4} shows that
$$
\lip(g,B(x_0,r)) \leq \lip(f_\lambda^\mu, B(x_0,r))+ \varepsilon \leq \lip(f, B(x_0,r+ \varepsilon))+ \varepsilon 
$$
for every ball $B(x_0,r)$ of $X.$ 
\end{proof}

Combining Lemma \ref{lemmapartialsmoothapproximation} with Theorem \ref{generaltheorem}, we obtain Theorem \ref{maintheoremnonseparablehilbert} and Theorem \ref{secondmaintheorem} when $X$ is a separable Hilbert space or a finite dimensional space.

\begin{remark}
\em{In the case when the function to be approximated vanishes on the boundary, the proof of Theorem \ref{secondmaintheorem} for finite dimensional spaces can be very much simplified as we do not need to use Theorem \ref{theoremglobalapproximation}. Indeed, if $\R^n$ is endowed with an arbitrary norm and $u_0: \overline{\Omega} \to \R$ is a Lipschitz function with $u_0=0$ on $\partial \Omega,$ given $\varepsilon >0,$ we define the function $\varphi_\varepsilon : \R \to \R$ by 
\begin{equation}\label{equationfunctioncomposition}
\varphi_\varepsilon(t)   =  \left\lbrace
	\begin{array}{ccl}
	t+\frac{\varepsilon}{2} & \mbox{if } & t \leq - \frac{\varepsilon}{2}, \\
	0  & \mbox{if }&  -\frac{\varepsilon}{2}\leq t\leq \frac{\varepsilon}{2}, \\
	t-\frac{\varepsilon}{2} & \mbox{if }& t \geq  \frac{\varepsilon}{2}.
	\end{array}
	\right.
\end{equation}
We can assume that $u_0$ is extended to all of $\R^n$ by putting $u_0=0$ on $\R^n \setminus \overline{\Omega},$ preserving the Lipschitz constant. The function $u = \varphi_\varepsilon \circ u_0$ defined on $\R^n$ is Lipschitz because so are $u_0$ and $\varphi_\varepsilon$, and $\lip(u,\R^n) \leq \lip(u_0, \R^n).$ Also, since $| \varphi_\varepsilon(t)-t| \leq \varepsilon/2$ for every $t\in \R,$ it is clear that
$$
| u(x)-u_0(x)| = | \varphi_\varepsilon(u_0(x))- u_0(x)| \leq \frac{\varepsilon}{2} \quad \text{for all} \quad x\in \R^d.
$$
Now we define
$$
v(x)=(u * \theta_\delta)(x)= \int_{\R^d} u(y) \theta_\delta(x-y) dy, \quad x\in \R^d,
$$
where $\theta_\delta : \R^d \to \R$ is a $C^\infty(\R^d)$ such that $\theta_\delta \geq 0, \: \int_{\R^d}\theta_\delta =1$ and $\sop(\theta_\delta) \subseteq B(0,\delta).$ Using the preceding remarks together with the well-known properties of the integral convolution of Lipschitz functions with mollifiers, it is straightforward to check that, for $\delta>0$ small enough, $v$ is the desired approximating function, i.e, $v$ is of class $C^\infty(\R^d)$ with $v=0$ on $\partial \Omega, \: \lip(v, \R^n) \leq \lip(u_0, \R^n)$ and $|u_0-v| \leq \varepsilon$ on $\overline{\Omega}.$ 

}
\end{remark}

\subsection{The space $c_0(\Gamma)$}

Let us now prove that the space $X= c_0(\Gamma)$ satisfies the hypothesis of Theorem \ref{generaltheorem} with $k=\infty.$ In order to do this, we will use the construction given in \cite[Theorem 1]{HajekJohanis2} and we will observe that the local Lipschitz constants are preserved. 

\begin{lemma}\label{lemmapartialc0}
If $\Gamma$ is an arbitrary subset, $X=c_0(\Gamma)$ and $f: X \to \R$ is a Lipschitz function, then, for every $\varepsilon >0,$ there exists a function $g: X \to \R$ of class $C^\infty(X)$ such that $|f-g| \leq \varepsilon$ on $X$ and $\lip(g,B(x_0,r)) \leq \lip( f, B(x_0, r+ \varepsilon))$ for every ball $B(x_0,r) \subset X.$
\end{lemma}
\begin{proof}
If $K$ denotes the Lipschitz constant of $f,$ let us consider $0 < \eta <\frac{\varepsilon}{2(1+K)}.$ Let us define the function $\phi: X \to X$ by $\phi(x) = \left(  \varphi_{2\eta}(x_\gamma) \right)_{\gamma \in \Gamma}$ for every $x=(x_\gamma)_{\gamma \in \Gamma} \in X,$ where $\varphi_{2\eta}$ is defined in \eqref{equationfunctioncomposition}. Thus $\phi$ is $1$-Lipschitz and satisfies $\| \phi(x)-x\| \leq \eta$ for every $x\in X.$ By composing $f$ with $\phi$ we obtain a function $h = f \circ \phi$ satisfying $|f-h| \leq \frac{\varepsilon}{2}$ and with the property that, for every $x\in X,$ there exists a finite subset $F$ of $\Gamma$ such that whenever $y,y'\in B(x, \frac{\eta}{2})$ and $P_F(y)=P_F(y')$ (here $P_F(z) = \sum_{\gamma \in F} e^*_\gamma(z) e_\gamma$ for every $z\in X$) we have $h(y)=h(y').$ Moreover, we observe that if $x,y\in B(x_0,r) \subset X,$ then $\phi(x), \phi(y) \in B(x_0, r+\eta)$ and therefore
$$
|h(x)-h(y)| \leq \lip( f, B(x_0,r+\eta) ) \| \phi(x)-\phi(y)\| \leq \lip( f, B(x_0,r+\eta) )  \| x-y\|;
$$
which shows that $\lip(h,B(x_0,r)) \leq \lip( f, B(x_0,r+\eta) ).$ Now we use the construction of \cite[Lemma 6]{HajekJohanis2} to obtain the desired approximation $g:$ let us define $g $ as the limit of the net $\lbrace g_F \rbrace_{F\in \Gamma^{<\omega}},$ where each $g_F$ is defined by
$$
g_F(x) =\int_{\R^{|F|}} h \Big ( x- \sum_{\gamma\in F} t_\gamma e_\gamma \Big ) \prod_{\gamma \in F} \theta(t_\gamma) d \lambda_{|F|}(t), \quad x\in X;
$$
and $\theta$ is a even $C^\infty$ smooth non-negative function on $\R$ such that $\int_{\R} \theta=1$ and $\sop(\theta) \subset [ - c \varepsilon, c \varepsilon],$ for a suitable small constant $c>0.$ It turns out that $g$ is of class $C^\infty(X)$ with $|g-h| \leq \frac{\varepsilon}{2}$ on $X$ and with the property that, for every $x\in X,$ there exists a finite subset $F_x$ of $\Gamma$ such that $g(x)=g_H(x)$ for every finite subset $H$ of $\Gamma$ containing $F_x.$ See \cite[Lemma 6]{HajekJohanis2} for details. In addition, we notice that if $x,y\in B(x_0,r),$ and we consider finite subsets $F_x$ and $F_y$ of $\Gamma$ with the above property, then for the set $H=F_x \cup F_y,$ we have that
\begin{align*}
| g(x)& -g(y)|  = |g_H(x)-g_H(y)|  \leq \int_{\R^{|H|}} \bigg | h \Big ( x- \sum_{\gamma\in H} t_\gamma e_\gamma \Big ) - h \Big ( y- \sum_{\gamma\in H} t_\gamma e_\gamma \Big ) \bigg | \prod_{\gamma \in H} \theta(t_\gamma) d \lambda_{|H|}(t) \\
 & \leq \lip( h, B(x_0, r+ c\varepsilon)) \| x-y\| \int_{\sop(\theta)^{|H|}}  \prod_{\gamma \in H} \theta(t_\gamma) d \lambda_{|H|}(t)= \lip( h, B(x_0, r+ c\varepsilon)) \| x-y\|.
\end{align*}
This shows that 
$$
\lip(g,B(x_0,r)) \leq \lip( h, B(x_0, r+ c\varepsilon))  \leq \lip( f, B(x_0,r+c\varepsilon + \eta) ),
$$
for every ball $B(x_0,r) \subset X.$ This proves the lemma.

\end{proof}

Combining Lemma \ref{lemmapartialc0} with Theorem \ref{generaltheorem}, we obtain Theorem \ref{secondmaintheorem} in the case $X=c_0(\Gamma).$ 

\section{Approximation by almost classical solutions of the Eikonal equation}\label{sectionapproximationalmostclassical}

Throughout this section $X$ will denote a finite dimensional normed space with $\dim(X) \geq 2.$ At the end of the section we will complete the proof of Theorem \ref{maintheoremarbitrarynorm}. 

\medskip

We need to recall the notion of \textit{almost classical solutions} of stationary Hamilton-Jacobi equations with Dirichlet boundary condition. This concept was introduced in \cite{DevilleMatheron} for the Eikonal equation and was generalized in \cite{DevilleJaramillo} as follows.

\begin{definition}
Let $\Omega$ be an open subset of $X$ and let $F : \R \times \Omega \times X^* \to \R$ and $u_0: \partial \Omega \to \R$ be continuous. A continuous function $u: \overline{\Omega} \to \R$ is an almost classical solution of the equation $F( u(x), x, Du(x))=0$ with Dirichlet condition $u=u_0$ on $\partial \Omega$ if:
\begin{itemize}
\item[$(i)$] $u=u_0$ on $\partial \Omega.$
\item[$(ii)$] $u$ is differentiable on $\Omega$ and $F( u(x), x, Du(x))\leq 0$ for all $x\in \Omega.$ 
\item[$(iii)$] $F( u(x), x, Du(x))=0$ for almost every $x\in \Omega.$ 
\end{itemize}
\end{definition}

In \cite[Theorem 4.1]{DevilleMatheron} it was proved the existence of almost classical solutions of the Eikonal equation with homogeneous boundary data, that is, $|D v| =1$ and $v=0$ on $\partial \Omega.$ This result was generalized in \cite{DevilleJaramillo} for an arbitrary function $F$ under certain conditions on $F.$ See \cite[Theorem 3.1]{DevilleJaramillo} or Proposition \ref{propositionexistencealmostclassicalsolution} below.

\medskip

We start by proving a slight refinement of \cite[Theorem 3.1]{DevilleJaramillo} for the existence of almost classical solutions, in which these solutions can be taken with arbitrarily small supremum norm.

\begin{proposition}\label{propositionexistencealmostclassicalsolution}
Let $\Omega \subset X$ be an open subset and let $F: \R \times \Omega \times X^* \to \R$ be a continuous mapping. Assume that
\begin{itemize}
\item[(A)] $F(0,x,0) \leq 0$ for every $x\in \Omega.$
\item[(B)] For every compact subset $K$ of $\Omega$ there exist constants $\alpha_K, M_K >0$ such that for all $x\in K, \: r\in [0, \alpha_K]$ and $ x^* \in X^*$ with $\| x^* \|_* \geq M_k$ we have $F(r,x,x^*)>0.$ 
\end{itemize}
Then, given $\varepsilon >0,$ there exists a function $u \geq 0$ on $\overline{\Omega}$ such that $| u |\leq \varepsilon$ on $\overline{\Omega}$ and $u$ is an almost classical solution of the equation $F(u(x),x,Du(x))=0$ on $\Omega$ with Dirichlet condition $u=0$ on $\partial \Omega.$ Moreover, the extension $\tilde{u}$ of $u$ defined by $\tilde{u}=0$ on $X \setminus \Omega$ is differentiable on $X.$ 
\end{proposition}
\begin{proof}
Although \cite[Theorem 3.1]{DevilleJaramillo} was originally stated when $X=\R^n$ is endowed with the euclidean norm, we can easily rewrite its statement (and its proof) for general finite dimensional normed spaces by using the following proposition, which is an easy consequence of \cite[Corollary 3.6]{DevilleMatheron}.

\begin{proposition}\label{propositionmonotonicbidualnorm}
Suppose that $B$ is a closed ball of $X^*.$ There exists a mapping $t : B \to S_{X^{**}}$ such that if $(\sigma_n)_n \subset B$ is a sequence with $t(\sigma_n)(\sigma_{n+1}-\sigma_n) \geq 0$ for every $n,$ then $(\sigma_n)_n$ converges.
\end{proposition}

In \cite[Theorem 3.1]{DevilleJaramillo}, $\Omega$ is decomposed as $\Omega = \bigcup_{j \geq 1} C_j,$ where $\lbrace C_j\rbrace_{j \geq 1}$ is a locally finite family of closed cubes and the function $u$ satisfies $u=0$ on $\bigcup_{j \geq 1} \partial C_j$ (because $u$ is the sum of a series of functions all vanishing on this union). Moreover, it is possible to choose the covering $\lbrace C_j\rbrace_{j \geq 1}$ so that $\diam(C_j) \leq \varepsilon$ for every $j \geq 1,$ and then, the Mean Value Theorem yields that $|u| \leq \varepsilon$ on $\Omega.$

\end{proof}

\begin{proof}[Proof of Theorem \ref{maintheoremarbitrarynorm}]
Given a $1$-Lipschitz function $u_0: \overline{\Omega} \to \R$ such that $u_0$ is $\lambda_0$-Lipschitz on $\partial \Omega $ for some $\lambda_0<1$ and given $\varepsilon >0,$ we can find, thanks to Theorem \ref{secondmaintheorem}, a $1$-Lipschitz function $v : \overline{\Omega} \to \R$ of class $C^\infty(\Omega)$ such that
\begin{equation}\label{estimationauxiliarfunction3}
|u_0-v| \leq \frac{\varepsilon}{2} \quad \text{on} \quad \overline{\Omega}, \quad  v=u_0 \quad \text{on} \quad \partial \Omega.
\end{equation}

Let us define $F: \Omega \times X^* \to \R$ by $F(x, x^*) = \| x^* +  D v (x) \|_*-1,$ for every $(x, x^*) \in \Omega \times X^* .$ Because $v$ is $1$-Lipschitz on $\overline{\Omega},$ we have $F(x,0) \leq 0$ for every $x\in \Omega$, which means that the function identically $0$ is a subsolution to the problem
\begin{equation}
\left\lbrace
	 \begin{array}{ccl}\label{particularHJequation}
	F(x,D u(x))=0 & \text{on } \Omega , \\
	u = 0 & \text{on } \partial \Omega,
	\end{array}
	\right.
\end{equation}
Also, observe that, whenever $\| x^*\|_*\ge 3$, we have, for all $x\in\Omega$,  $F(x, x^*) \ge 1$.  Hence, 
Proposition \ref{propositionexistencealmostclassicalsolution} provides an almost classical solution $u$ to problem 
\eqref{particularHJequation} such that $| u| \leq \varepsilon/2$ on $\overline{\Omega}.$ Let us define $w = u+ v$ on 
$\overline{\Omega}.$ Then $w$ is continuous on $\overline{\Omega}$ and differentiable on $\Omega$ with $\| Dw(x)\|_* = \| D 
u(x)+ Dv(x)\|_* \leq 1$ for every $x\in \Omega$ and $\| Dw(x)\|_*=1$ for almost every $x\in \Omega.$ Also, $w$ satisfies that 
$w=v=u_0$ on $\partial \Omega$ and $|w-v| \leq \varepsilon/2$ on $\overline{\Omega}.$ Using Fact 
\ref{factlipschitzconstant}, we obtain that $w$ is in fact $1$-Lipschitz on $\overline{\Omega}.$ Finally note that
$$
|u_0-w| \leq |v-w| + |u_0-v|    \leq \frac{\varepsilon}{2} + \frac{\varepsilon}{2} \leq \varepsilon \quad \text{on} \quad \overline{\Omega}.
$$
This completes the proof of Theorem \ref{maintheoremarbitrarynorm}.

\end{proof}

\section{The limiting case}\label{sectionexample}

In this section we are concerned about constructions of functions $u_0$ with prescribed values on the boundary of $\Omega$ such that $u_0$ is differentiable on $\Omega$ and $\lip(u_0, \partial \Omega) = \lip(u_0, \Omega).$ 

\begin{proposition}\label{propositiontwodimensions}
If $\Omega \subset \R^2$ is open and $u_0: \partial \Omega \to \R$ is $1$-Lipschitz for the usual euclidean distance, then there exists a differentiable $1$-Lipschitz function $w: \overline{\Omega} \to \R$ such that $| \nabla w| =1$ almost everywhere on $\Omega$ and $w=u_0$ on $\partial \Omega,$ i.e, there exist almost classical solutions of the Eikonal equation with boundary value $u_0.$  
\end{proposition}
\begin{proof}
We know by O. Savin's results in \cite{Savin} that the \textit{Absolutely Minimizing Lipschitz Extension} (AMLE for short) of $u_0$ to $\overline{\Omega}$ is of class $C^1(\Omega).$ In particular, there exists a $1$-Lipschitz extension $v: \overline{\Omega} \to \R$ of $u_0$ such that $v \in C^1(\Omega).$ If we consider the problem
\begin{equation}\label{auxiliarhomogeneousequation}
 \left\lbrace
	 \begin{array}{ccl}
	| \nabla u + \nabla v  |=1 & \text{on } \Omega , \\
	u = 0 & \text{on } \partial \Omega,
	\end{array}
	\right.
\end{equation}
and define $F: \Omega \times \R^2 \to \R$ by $F(x,p) = | p + \nabla v(x)|, \: x\in \Omega,\: p \in \R^2,$ we have that $F$ is a continuous function which is easily checked to satisfy the hypothesis of \cite[Theorem 3.1]{DevilleJaramillo} (see Proposition \ref{propositionexistencealmostclassicalsolution} in Section \ref{sectionapproximationalmostclassical}) for the existence of an almost classical solution to the problem \eqref{auxiliarhomogeneousequation}. If we denote by $u$ this solution and we set $w = u+v$ on $\overline{\Omega},$ it is clear that $w$ is the desired function.
\end{proof}

We notice that the proof of Proposition \ref{propositiontwodimensions} cannot be adapted for dimension $n \geq 3,$ because it is unknown whether or not the AMLE of $u_0$ is of class $C^1.$ We only know from the results in \cite{EvansSmart}, that these AMLE are differentiable everywhere. 

\begin{example}\label{counterexamplel1}
{\em Consider the $\ell_1$ norm on $\R^2$ and define $\Omega = \lbrace (x,y) \in \R^2 \: : \: x^2+y^2<1 \rbrace$ and the function $u_0(x,y)= |x|-|y|$ on the boundary $\partial \Omega$ of $\Omega.$ The function $u_0$ is $1$-Lipschitz and all possible $1$-Lipschitz extensions of $u_0$ to $\overline{\Omega}$ are not differentiable at $(0,0).$}
\end{example}
\begin{proof}
Given $(x,y), (x',y') \in \partial \Omega,$ we can easily write
$$
| u(x,y)-u(x',y')| = \big | |x|-|x'| + |y'|-|y| \big | \leq |x-x'| + |y-y'| = \| (x,y)-(x',y')\|_1,
$$
where the above inequalities are sharp. Thus, $u_0$ is a $1$-Lipschitz function on $\partial \Omega.$ Now, let $u : \overline{\Omega} \to \R$ be a $1$-Lipschitz extension of $u_0.$ We have that $u(0,0)\leq 0$ since $ u(0,0)+1  = u(0,0)-u(0,1) \leq 1.$ On the other hand, for every $x\in [-1,1],$ we can write
\begin{align*}
 u(x,0)  &  \geq u(\sign(x),0) - \|(\sign(x),0)-(x,0)\|_1 = 1 -(1-|x|) = |x|\\
  u(x,0) &  \leq u(0,0)+ \|(x,0)-(0,0)\|_1 \leq |x| ;
\end{align*}
which implies that $u(x,0)=|x|$ for every $x\in [-1,1].$ Therefore $u$ is not differentiable at $(0,0).$ 

\end{proof}

The above example shows in particular that if $u_0$ is extended to a $1$-Lipschitz on $\overline{\Omega}$ and $\varepsilon>0,$ there is no $1$-Lipschitz function $v$ on $\overline{\Omega}$ which is differentiable on $\Omega, \: v=u_0$ on $\partial \Omega$ and $|u_0-v| \leq \varepsilon$ on $\overline{\Omega}.$ Thus Problem \ref{mainproblem} has a negative solution in the limiting case $\lip(u_0, \partial \Omega) = \lip(u_0,\overline{\Omega}).$ An example with the same properties can be obtained with the $\ell_\infty$ norm by means of the isometry $T : (\R^2, \| \cdot \|_1)  \to (\R^2, \| \cdot \|_\infty), \: T(x,y)=(x+y,x-y).$

\section*{Acknowledgements}
The second author wishes to thank the Institut de Math\'{e}matiques de Bordeaux, where this work was carried out.

\end{document}